\documentclass{article}

\usepackage[latin1]{inputenc}
\usepackage[T1]{fontenc}
\usepackage{indentfirst}
\usepackage{amsmath,enumerate}
\usepackage{amsfonts,amssymb,latexsym,xspace}
\usepackage{amsthm}
\usepackage{amscd}
\usepackage{dsfont}
\usepackage{verbatim}
\usepackage{graphicx}
\usepackage{amsfonts,latexsym,epsf}
\usepackage{arabtex}
\usepackage[all]{xy}
\usepackage{mathrsfs}
 \numberwithin{equation}{section}
\usepackage{setspace}

\newtheorem{thm}{\textbf{Theorem}}[section]
\newtheorem{df}[thm]{\textbf{Definition}}
\newtheorem{lemme}[thm]{\textbf{Lemma}}
\newtheorem{proposition}[thm]{\textbf{Proposition}}
\newtheorem{corollaire}[thm]{\textbf{Corollary}}
\newtheorem{ex}[thm]{\textbf{Exemple}}
\newtheorem{prop}[thm]{\textbf{Proposition}}
\newtheorem{rem}[thm]{\textbf{Remark}}
\newtheorem*{CB}{\textbf{Conjectures of Beilinson}}

\def\CH{\mathop{\rm CH}\nolimits}
\def\Corr{\mathop{\rm Corr}\nolimits}
\def\End{\mathop{\rm End}\nolimits}

\title{Motives of quadric bundles}

\author{Johann Bouali}

\begin{document}

\maketitle

\begin{abstract}
This article is about motives of quadric bundles. In the case of odd dimensional fibers and where the basis is of dimension two we give an explicit relative and absolute Chow-K\"unneth decomposition. This shows that the motive of the quadric bundle is isomorphic to the direct sum of the motive of the base and the Prym motive of a double cover of the discriminant. In particular this is a refinement with $\mathbb Q$ coefficients of a result of Beauville concerning the cohomology and the Chow groups of an odd dimensional quadric bundle over $\mathbb P^2$. This Chow-K\"unneth decomposition satisfies Murre's conjectures II and III. This article is a generalization of an article of Nagel and Saito on conic bundles \cite{NS}.
\end{abstract}

\section{Introduction}

In this article we work with algebraic varieties over the field $\mathbb C$ of complex numbers, i.e. schemes of finite type over $\mathbb C$. All the Chow groups of these varieties will be with rational coefficients. All the singular cohomology groups will be with rational coefficients.

Murre's conjectures give an answer of Beilinson's conjectures on the Chow groups of the smooth projective varieties.

\begin{CB} For a smooth projective variety $X$ of dimension $n$, there exists on $\CH^l(X)$, $0\leq l\leq n$,
a decreasing filtration $F^\nu$ ($\nu\leq 0$) with the following properties.

\begin{itemize}
\item[(i)] $F^0=\CH^l(X)$, $F^1=\CH^l_{\hom}(X)$.
\item[(ii)] $F^r\cdot F^s \subset F^{r+s}$ under the intersection product.
\item[(iii)] $F^*$ is functorial for morphisms $f: X\to Y$.
\item[(iv)] Assuming that the K\"unneth components $[\Delta_X]^{2n-j,j}\in F^nH^{2n}(X\times X,\mathbb{Q})$ of the cohomology class of the diagonal
are algebraic, given by the class of $p_j\in Z^n(X\times X,\mathbb{Q})$, then $Gr_{F^*}^\nu\CH^l(X)$ depends only on the motive $h^{2l-\nu}(X)=(X,[p_{2l-\nu}]_{\hom})$ modulo homological equivalence, i.e.
$p_{j,*}Gr_{F^*}^\nu\CH^l(X)=\delta_{j,2l-\nu}Id_{Gr_{F^*}^\nu\CH^l(X)}$, a correspondence $\Gamma\in \CH^n_{\hom}(X\times X)$ acting as zero on $Gr_{F^*}^\nu\CH^l(X)$(see (i) and (ii) for $X\times X$ and (iii)).
\item[(v)] $F^m=0$ for $m>>0$  (weak version) or, $F^{l+1}=0$ (strong version). 
\end{itemize}

\end{CB}

\begin{df}(\cite[sec. 4.2.3]{MS-P}, \cite[def. 6.1.1]{livre motif})

Let $X$ be a smooth projective variety of dimension $n$. We say that $X$ has a Chow-K\"unneth decomposition if there exists, for $0\leq j\leq 2n$, $p_j\in{\CH_n(X\times X)}$ such that 
\begin{enumerate}
\item $p_jp_k=\delta_{j,k}p_j$ i.e. $p_j$ are projectors, mutually orthogonal,
\item $\sum_{j=0}^{2n}p_j=\Delta_X\in{\CH_n(X\times X)}$ where $\Delta_X$ denotes the class of the diagonal $\Delta_X\subset{X\times X}$
\item $[p_j]_{hom}=[\Delta_X]^{2n-j,j}$ where $[\Delta_X]^{2n-j,j}\in F^nH^{2n}(X\times X,\mathbb{Q})$ denotes the $j^{th}$ K\"unneth component of the cohomology class of the diagonal. Equivalently, the action of $p_j$ on cohomology is $p_{j,*}H^k(X)=\delta_{j,k}I_{H^k(X)}$.
\end{enumerate}
\end{df}

\begin{df}(\cite[sec. 4.3.2.1]{MS-P}, \cite[sec. 7.2]{livre motif})

For a smooth projective variety $X$ of dimension $n$, Murre stated the following set of conjectures :
\begin{enumerate}
\item [I] (Existence) $X$ admits a Chow-K\"unneth decomposition given by projectors $p_j$, for $0\leq j\leq 2n$.
\item [II] For all $0\leq l\leq n$, the projectors $p_0$,..., $p_{l-1}$ and $p_{2l+1}$,..., $p_{2n}$ act as zero on $\CH^l(X)$.
\item [III] For all $0\leq l\leq n$, $\ker(p_{2l})=\CH^l(X)_{hom}$. (Note that clearly $\ker(p_{2l})\subset{\CH^l(X)_{hom}}$).
\item [IV] For all $0\leq l\leq n$ the filtration on $\CH^l(X)$ defined by $F^p\CH^l(X)=\ker(p_{2l})\cap \ker(p_{2l-1})\cap ...\cap \ker(p_{2l-p+1})$ is independent of the choice of Chow-K\"unneth decomposition.
\end{enumerate}
\end{df}
Jannsen proved \cite[Thm. 5.2]{Ja94}
that this set of conjectures is equivalent to the strong version of conjectures of Beilinson \cite[sec. 4.3.2.2]{MS-P} and that if the
conjectures are true, the filtrations agree.

\begin{ex}

\begin{itemize}
\item[1.] Smooth projective curves satisfy Murre's conjectures.
\item[2.] For smooth projective surfaces, Murre has constructed a Chow-K\"unneth decomposition, via Picard and Albanese projectors which satisfy Murre's conjectures II and III (\cite{Murre},\cite{Murre2}).
\item[3.] Smooth complete intersections in the projective space admit a Chow-K\"unneth decomposition.
\item[4.] If a smooth projective variety $X$ of dimension $n$ admits a Chow-K\"unneth decomposition given by projectors $p_0,\ldots,p_{2n}$,
then $X\times X$ admits a Chow-K\"unneth decomposition, given by the projectors $q_r=\sum_{i+j=r}p_i\times p_j\in \CH^{2n}(X\times X \times X \times X)$, for $0\leq r\leq 4n$.   
\item[5.] Abelian varieties admit a Chow-K\"unneth decomposition (\cite{Sh74},\cite{DM91}).
\item[6.] Uniruled threefolds admit a Chow-K\"unneth decomposition (\cite{AMS98}) and also certain classes of threefolds with a special condition on $H_{trans}^2(X)$ admit a Chow-K\"unneth decomposition (\cite{AMS00},\cite{MS-S}).
\item[7.] Elliptic modular varieties admit a Chow-K\"unneth decomposition (\cite{GHM}).
\end{itemize}
\end{ex}

\begin{rem}
If a smooth projective variety $X$ a dimension n admits a Chow-K\"unneth decomposition given by projectors
$p_j\in{\CH_n(X\times X)}$ then we have an isomorphism of Chow motives :
$$ch(X)=(X,\Delta_X)\simeq\oplus_{j=0}^{2n}(X,p_j).$$
\end{rem}


Given a quasi-projective variety $S$, let $\mathcal{V}(S)$ be the category whose objects are smooth varieties $X$ with a proper morphism
$f:X\to S$. For $X$, $Y\in{\mathcal{V}(S)}$ a relative correspondence from $X$ to $Y$ is an element $\Gamma\in{\CH(X\times_S Y)}$. We say that $\Gamma$ has degree $p$ if $\Gamma\in\CH_{d_Y-p}(X\times_S Y)$.

The composition of relative correspondences is defined as follows :
\begin{df}(Corti-Hanamura \cite{Corti-Hanamura})

Let $X$,$Y$,$Z\in{\mathcal{V}(S)}$. 
Let $\Gamma_1\in{\CH(X\times_S Y)}$, $\Gamma_2\in{\CH(Y\times_S Z)}$.
We then define
$$\Gamma_2\circ\Gamma_1=p_{X,Z,*}\delta^![\Gamma_1\times\Gamma_2]$$ 
where $\delta=\Delta_Y:Y\hookrightarrow Y\times Y$ is the diagonal embedding and $p_{X,Z}:X\times_S Y\times_S Z\rightarrow X\times_S Z$ is the projection.
Since $Y$ is smooth, $\delta:Y\hookrightarrow Y\times Y$ is a regular embedding (local complete intersection), hence $\delta^!$ is well defined (see \cite{W.Fulton}). Moreover $p_{X,Z}$ is proper.  
\end{df}
 
A relative projector on $X\in\mathcal{V}(S)$ is a relative correspondence of degree $0$ from $X$ to $X$ satisfying $p\circ p=p$.
The category of Chow motives over $S$ is the category whose objects are triples $(X,p,m)$, with $X\in\mathcal{V}(S)$, $p$ a relative projector on $X$
and a morphism from $(X,p,m)$ to $(Y,q,n)$ is a relative correspondence of degree $n-m$ from $X$ to $Y$.

As composition of relative correspondences is compatible with flat base change $S'\rightarrow S$ \cite[Lem. 8.1.6]{livre motif}, there exists a functor
from the category of Chow motives over $S$ to the category of (absolute) Chow motives.

By \cite{BBDG} we have a decomposition 
$$Rf_*\mathbb Q\simeq\oplus_i ^pR^if_*\mathbb Q[-i]$$
in the derived category $\mathcal{D}^b_c(S)$ of constructible sheaves of $\mathbb Q$-vector spaces over $S$.
where $^pR^jf_*{\mathbb Q}_X$ denotes the $j^{th}$ perverse cohomology of $Rf_*{\mathbb Q}_X$.

Note that a relative projector $q$ on $X$ acts on the perverse cohomology $^pR^jf_*{\mathbb Q}_X$ by \cite{Corti-Hanamura}. We denote the action by $q_*$.

\begin{df}(\cite[sect. 4.4.2.2]{MS-P}, \cite[Def. 8.3.3]{livre motif})

Let $f:X\rightarrow S$ be a projective morphism with $X$ smooth of dimension $n$. Let $k$ be the dimension of the generic fiber.
 We say that $f$ has a relative Chow-K\"unneth decomposition if there exists $p_i\in{\CH_n(X\times_S X)}$, $0\leq i\leq 2k$, such that
\begin{enumerate}
\item $p_ip_j=\delta_{i,j}p_i$ ie $p_i$ are relative projectors, mutually orthogonal,
\item $\sum_i p_i=\Delta_X\in{\CH_n(X\times_S X)}$
\item $p_{i,*}{^p}R^jf_*{\mathbb Q}_X=\delta_{i,j}I_{^pR^jf_*{\mathbb Q}_X}$.
\end{enumerate}
\end{df}

\begin{rem}
Let $f:X\rightarrow S$ be a projective morphism with $X$ smooth of dimension $n$. Let $k$ be the dimension of the generic fiber.
If $f$ has a relative Chow-K\"unneth decomposition given by relative projectors $p_j\in{\CH_n(X\times_S X)}$, $0\leq j\leq 2k$ then
we have an isomorphism of relative Chow motives :
$$ch(X/S)=(X,\Delta_X)\simeq\oplus_{j=0}^{2k}(X,p_j).$$
\end{rem}

\begin{df}
Let $X$ and $S$ two algebraic varieties of dimension n and r over $\mathbb{C}$, $X$ is a quadric bundle over $S$ if there exists a flat projective morphism $f:X\rightarrow S$, such that the fibers of $f$ are quadrics of dimension $n-r$.
\end{df}

Here we say that the morphism is projective if we have a relative embedding $X\hookrightarrow\mathbb P(\mathcal{E})\rightarrow S$, with 
$\mathbb P(\mathcal{E})\rightarrow S$ a projective bundle of dimension $n+1$.

We refer to \cite{G.H} for the elementary proprieties of quadrics.

\begin{df}
Let $\xi_i\subset{X}$ be a relative linear section of relative dimension $i$ of 
$f:X\hookrightarrow\mathbb P(\mathcal{E})\rightarrow S$. We define
$$p_{2i}=1/2[\xi_i\times_S\xi_{n-r-i}]\in{\CH_n(X\times_S X)}.$$
\end{df}

The base of the quadric bundles admits a natural stratification,

\begin{proposition}(\cite{T.Terasoma})
Let $f:X\hookrightarrow\mathbb P(\mathcal{E})\rightarrow S$ be a quadric bundle, there is a stratification $\Delta_k\subset{...}\subset{\Delta_1}\subset{S}$ of $S$ by closed subsets where 
$$\Delta_k=\left\{s\in{S},\, rk(X_s)\leq n-r+2-k\right\}.$$
Moreover, if $S$ is smooth we have the inclusion $\Delta_{k+1}\subset{sing(\Delta_k)}$ by properties of determinantal subvarieties\cite{S.Mukai}.

\end{proposition}

In this paper we consider a quadric bundle over a surface of odd relative dimension $2m-1$.
Consider the Stein factorization $F_m(X/S)\rightarrow\tilde{\Delta}\rightarrow\Delta$ where $F_m(X/S)$ denotes the relative Fano variety of $m$-dimensional linear subspaces and $\Delta=\Delta_1\subset{S}$. It is a double covering ramified over $\Delta_2\subset{\Delta}$. 

We will note $\Gamma\subset{F_m(X/S)\times_S X}$ the incidence correspondence with $p:\Gamma\rightarrow F_m(X/S)$, $q:\Gamma\rightarrow X$ the restriction to $\Gamma$ of the two projections. 

The morphism $F_m(X/S)\rightarrow\tilde{\Delta}$ is flat projective, so we can take $W\subset{F_m(X/S)}$ be a multisection of $F_m(X/S)\rightarrow\tilde{\Delta}$. We note $\Gamma_W\subset{W\times_S X}$ the restriction of $\Gamma$ to $W$ and $Z=q(\Gamma_W)\subset{X}$. We have $dim(Z\times_{\Delta}Z)=n$ with $Z\times_{\Delta}Z=(Z\times_{\Delta}Z)^+\cup(Z\times_{\Delta}Z)^-$ two $n$ dimensional components :
$$(Z\times_{\Delta}Z)^+=(q\times q)(p\times p)^{-1}(g\times g)^{-1}(\Delta_{\tilde{\Delta}}),~
(Z\times_{\Delta}Z)^-=(q\times q)(p\times p)^{-1}(g\times g)^{-1}(D^{-})$$
where $g:W\rightarrow\tilde{\Delta}$ is the covering and $D^{-}$ denotes the graph of the involution.

\begin{df}

Let $W\subset{F_m(X/S)}$ be a multisection of degree $2d$ and $Z=q(\Gamma_W)\subset{X}$. We define
$$p_{\infty}=(-1)^m1/d^2([(Z\times_{\Delta}Z)^+]-[(Z\times_{\Delta}Z)^-])\in{\CH_n(X\times_S X)}.$$

\end{df}

The main result of the paper is the following :

\begin{thm}
\label{main}

Let $f:X\rightarrow S$ be a quadric bundle with odd dimensional fibers $2m-1$, $X$ smooth projective, $S$ smooth projective with $dim(S)=2$. For more simplicity we assume that 
\begin{itemize}
\item $C=\Delta_1$ is smooth irreducible,
\item The double covering $\tilde{C}\rightarrow C$ is non trivial.
\end{itemize}
Then
\begin{enumerate}
\item [(1)] The orthonormalization $(\{p_{2i}^N\}_{0\leq i\leq 2m-1}, p_{\infty}^N)$ of the family $(\{p_{2i}\}_{0\leq i\leq 2m-1}, p_{\infty})$ constitute a relative Chow-K\"unneth decomposition of $f:X\rightarrow S$.
\item [(2)] This relative Chow-K\"unneth decomposition induces an absolute Chow-K\"unneth decomposition of $X$ which satisfies Murre's conjectures I, II and III.
\end{enumerate}
\end{thm}
   
\begin{rem}   

A similar result has been obtained independently by C.Vial\cite{C.Vial2} using a different approach(he works only with absolute motives). 
In the case of quadric bundle, our result is more precise since it provides an explicit construction of the Prym motive.
  
\end{rem}   
   
The proof is based on a generalization of the techniques of \cite{NS}.
It consists of the following steps :
\begin{itemize}
\item We show that the $p_{2i}$ and $p_{\infty}$ are relative projectors, for $0\leq i\leq 2m-1$.
\item Using the Gramm-Schmidt process we obtain mutually orthogonal relative projectors $p_{2i}^N$ and $p_{\infty}^N$, for $0\leq i\leq 2m-1$.
\item We show that these relative projectors act in the right on the perverse direct images.
\item We verify the identity $\Delta_X=\sum_{i=0}^{2m-1} p_{2i}^N+p_{\infty}^N\in{\CH_n(X\times_S X)}$.
\item We pass to the corresponding absolute projectors and decompose then to obtain an absolute Chow-K\"{u}nneth decomposition of $X$ which satisfies Murre's conjectures I, II and III.
\end{itemize}

The paper is organized as follows. 
In section 2 we prove steps 1,2 and 3 (the proof that $p_{\infty}$ is a relative projector uses the result from the Appendix 5.1).
Section 3 is devoted to the proof of step 4. The idea of the proof is as follows. Put 
$$R=\Delta_X-\sum_{i=0}^{2m-1} p_{2i}^N-p_{\infty}^N\in{\CH_n(X\times_S X)}.$$ Consider the localization exact sequence
$$\CH_n(X_{C}\times_{C}X_{C})\stackrel{(i\times i)_*}{\rightarrow}\CH_n(X\times_S X)\stackrel{(j\times j)^*}{\rightarrow}\CH_n(X_U\times_U X_U)\rightarrow 0.$$
We first show that $R|_{X_U\times_U X_U}=0$. Hence $R$ is the image of an element $T\in{\CH_n(X_{C}\times_{C}X_{C})}$.
We then prove that $R$ is nilpotent using a description of $\CH_n(X_{C}\times_{C}X_{C})$. Hence $R=0$ since it is a projector.
The description of $\CH_n(X_{C}\times_{C}X_{C})$ represents the main technical difficulty of the paper ; in the case of conic bundles this step is easier since  
$\CH_3(X_{C}\times_{C}X_{C})$ is generated by the irreducible components of $X_{C}\times_{C}X_{C}$.
\\

This article is a part a my PhD thesis under the supervision of D.Markushevich and J.Nagel.

\section{Relative Chow-K\"unneth decomposition I}

Let $f:X\rightarrow S$ be a quadric bundle with odd-dimensional fibers of dimension $2m-1$, $X$ and $S$ smooth projective with $dim(S)=2$.
Then $dim(X)=n=2m+1$.
Assume $C=\Delta_1$ is smooth so that $\Delta_2=\emptyset$.

Consider the commutative diagram
$$\xymatrix{X \ar[rr]^i\ar[dr]^f& &\mathbb P(E)\ar[dl]^g\\
&S&}$$

\begin{proposition}
The BBDG decomposition\cite{BBDG} of $Rf_*\mathbb Q$ is given by
$$Rf_*\mathbb Q=\oplus_{j=0}^{2m-1}\mathbb Q_S[-2j]\oplus i_{C*}\mathcal{V}[-2m]\in{\mathcal{D}^b_c(S)}$$
where
$\mathcal{V}=i^{-1}(R^{2m}f_*\mathbb Q)_v$ is 
the cokernel of the sheaf morphism $i^*:R^{2m}g_*\mathbb Q\rightarrow R^{2m}f_*\mathbb Q$
induced by the above diagram.
The sheaf $\mathcal{V}$ is a local system of rank one on $C$.

\end{proposition}

\begin{proof}
This follows from the perverse versions of the Lefschetz hyperplane theorem and the hard Lefschetz theorem ;
cf.\cite[Proposition 8.5.2]{livre motif}
\end{proof}

To obtain a relative Chow-K\"unneth decomposition we need to define a family of mutually orthogonal relative projectors
that induce the above decomposition and add up to the identity.

Let for $0\leq i\leq 2m-1$, $\xi_i\subset{X}$ be a relative linear section of relative dimension $i$ of
$X\hookrightarrow\mathbb P(\mathcal{E})\rightarrow S$.

Consider the relative correspondences $p_{2i}=1/2[\xi_i\times_S\xi_{2m-1-i}]$. By Lieberman's Lemma \cite{livre motif}
we have $p_{2i}=1/2[\xi_{2m-1-i}][\xi_i]^t$.

\begin{prop}

The relative correspondences $p_{2i}$ are projectors and satisfy
\begin{enumerate}[(i)]
\item $p_{2j}p_{2i}=0$ if $j>i$
\item $p_{2j}p_{2i}=0$ if $j<i-2$ 
\end{enumerate}

\end{prop}

\begin{proof}

We have the composition of relative correspondences
$$p_{2i}=1/2[\xi_{2m-1-i}][\xi_i]^t$$ with 
$[\xi_i]^t\in{\CH_{i+r}(X)=\CH_{i+r}(X\times_S S)=Cor_S(X,S)}$ 
\\and $[\xi_{2m-1-i}]\in{\CH_{n-i}(X)=\CH_{n-i}(S\times_S X)=Cor_S(S,X)}$.

Thus
$$p_{2j}p_{2i}=1/4[\xi_{2m-1-j}][\xi_j]^t[\xi_{2m-1-i}][\xi_i]^t$$

We have $[\xi_j]^t[\xi_{2m-1-i}]\in{\CH_{r+j-i}(S\times_S S)=\CH_{r+j-i}(S)}$.
\\So 
$[\xi_j]^t[\xi_{2m-1-i}]=0$ if $j>i$ or $j<i-r$.
Moreover
$[\xi_j]^t[\xi_{2m-1-i}]=2[S]$ if $i=j$.

The result then follows by the associativity of the composition of relative correspondences.

\end{proof}

Consider the Stein factorization $F_m(X/S)\rightarrow\tilde{C}\rightarrow C$ where $F_m(X/S)$ denotes the relative Fano variety of $m$ linear subspaces and $C=C_1\subset{S}$. It is a non trivial \'etale double covering by the hypotheses of Theorem \ref{main} 
Let $\tau:\tilde{C}\rightarrow\tilde{C}$ be the involution of the double covering. 

Let $\Gamma\subset{F_m(X/S)\times_S X}$ be the incidence correspondence with $p:\Gamma\rightarrow F_m(X/S)$, $q:\Gamma\rightarrow X$ the restriction to $\Gamma$ of the two projections. 

For $W\subset{F_m(X/S)}$ a multisection of degree $d$ of $F_m(X/S)\rightarrow\tilde{C}$ we note $g:W\rightarrow\tilde{C}$ the covering of degree $d$,
$\Gamma_W\subset{W\times_S X}$ the restriction of $\Gamma$ to $W$ and $Z=q(\Gamma_W)\subset{X}$.

We have
$$\tilde{C}\times_S\tilde{C}=D^+\sqcup D^-$$
where $D^+=\Delta_{\tilde{C}}$ and $D^-=\tau^*$.
We will note $\rho=I-\tau_*\in{\Corr_C(\tilde{C},\tilde{C})}$ the prym motive.

\begin{rem}
If the covering $\tilde{C}\rightarrow C$ is trivial $\tilde{C}$ and $W$ have two irreducible and connected components(the covering is non ramified) and $\tilde{C}\times_S\tilde{C}$ has four connected components $D^+_1$, $D^+_2$, $D^-_1$, $D^-_2$.
\end{rem}

The variety
$Z\times_C Z\subset{X\times_S X}$
has two irreducible components
$(Z\times_{C}Z)^{+,-}=(q\times q)(p\times p)^{-1}(g\times g)^{-1}(D^{+,-})\subset{X\times_S X}$.

Note $$\gamma=\Gamma_Wg^*\in{\Corr_C(\tilde{C},X)}.$$

We will note:
$$q_{m,C}^+=\gamma\circ\gamma^t=[(Z\times_{C}Z)^+]\in{\CH_n(X\times_S X)},$$
$$q_{m,C}^-=\gamma\circ\tau_*\circ\gamma^t=[(Z\times_{C}Z)^-]\in{\CH_n(X\times_S X)}.$$

We will consider 
$$p_{\infty}=(-1)^m1/d^2\gamma\circ(I-\tau_*)\circ\gamma^t=(-1)^m1/d^2(q_{m,C}^+-q_{m,C}^-)\in{\CH_n(X\times_S X)}.$$

Note that by construction 
\begin{eqnarray}\label{transposition}
^tp_{\infty}=p_{\infty},\ \ 
^tp_{2i}=p_{4m-2-2i}.
\end{eqnarray}

\begin{prop}
\label{ortho1}
\begin{enumerate}
\item[(i)] $p_{\infty}$ is a relative projector
\item[(ii)] $p_{2i}p_{\infty}=0$ for $i\neq m-2$
\item[(iii)] $p_{\infty}p_{2i}=0$ for $i\neq m+1$.
\end{enumerate}

\end{prop}

\begin{proof}

Part (i) follows by the Appendix \ref{app}(1) :

$$p_{\infty}p_{\infty}=(1/d^2)^2\gamma\circ(I-\tau_*)\circ\gamma^t\circ\gamma\circ(I-\tau_*)\circ\gamma^t
=(-1)^m 1/d^2\gamma\circ(I-\tau_*)\circ\gamma^t=p_{\infty}$$
since $\gamma^t\circ\gamma=(-1)^m(I-\tau_*)\in{\CH_1(\tilde{C}\times_C\tilde{C})}$.

For the proof of part (ii), note that $p_{2i}p_{\infty}=0$ if $i\neq m-1$, $i\neq m-2$ for dimensional reasons.
Specifically, $p_{2i}q_{m,C}^+=p_{2i}q_{m,C}^-=0$ if $i\neq m-1$, $i\neq m-2$ since
$$p_{2i}q_{m,C}^+=1/2[\xi_{2m-1-i}]([\xi_i]^t\gamma)\gamma^t$$
and $[\xi_i]^t\gamma\in{\CH_{i-m+2}(\tilde{C}\times_S S)=\CH_{i-m+2}(\tilde{C})}$.

For $i=m-1$ we have $[\xi_i]^t\gamma=\lambda[\tilde{C}]$ with $\lambda\in{\mathbb Z}$, hence $p_{2m-2}p_{\infty}=0$ since
$\tau^*[\tilde{C}]=[\tilde{C}]$.

Part (iii) follows from part (ii) by transposition using (\ref{transposition}).

\end{proof}

We now use the Gramm-Schmidt orthonormalization process :

\begin{lemme}\cite[Lemma 4.11]{C.Vial}
\label{GS}

Let $V$ be a $\mathbb Q$-algebra and $k$ a positive integer. Let $p_0$,....,$p_n$ idempotents of $V$ which satisfy $p_jp_i=0$ for $j>i$. Then
$$\tilde{p_i}=(1-1/2p_0)...(1-1/2p_{i-1})p_i(1-1/2p_{i+1})...(1-1/2p_n)$$ define idempotents so that 
$\tilde{p_j}\tilde{p_i}=0$ for $j>i$ and $j=i-1$.
Moreover if we apply this process $r$ times we obtain
idempotents $p^N_i$ which satisfy $p^{N}_jp^{N}_i=0$ for $j>i$, and $j<i-r$.

\end{lemme}

Applying the Lemma two times to the family $\{p_{2i}\}_{0\leq i\leq 2m-1} $
we obtain idempotents $\{\tilde{p}_{2i}\}_{0\leq i\leq 2m-1}$ such that $\tilde{p}_{2i}\tilde{p}_{2j}=0$ for $i\neq j$. 
These projectors satisfy again :

\begin{prop}
\label{ortho2}
\begin{enumerate}
\item[(i)] $\tilde{p_{2i}}p_{\infty}=0$ for $i\neq m-2$.
\item[(ii)] $p_{\infty}\tilde{p_{2i}}=0$ for $i\neq m+1$.
\end{enumerate}

\end{prop}

\begin{proof}

Part (i) follows from Proposition 2.4 and by noticing that $p_{2m-6}p_{2m-4}p_{\infty}=0$ for dimensional reasons.
More precisely, $p_{2m-6}p_{2m-4}q_{m,C}^+=p_{2m-6}p_{2m-4}q_{m,C}^-=0$ since
$$p_{2m-6}p_{2m-4}q_{m,C}^{+}=1/2[\xi_{m+2}]([\xi_{m-3}]^t[\xi_{m+1}][\xi_{m-2}]^t\gamma)\gamma^t$$
and $[\xi_{m-3}]^t[\xi_{m+1}][\xi_{m-2}]^t\gamma\in{\CH_{-1}(\tilde{C}\times_S S)=\CH_{-1}(\tilde{C})}=0$.

Part (ii) follows from part (i) by transposition using (\ref{transposition}).

\end{proof}

Using Proposition \ref{ortho2}, we can apply Lemma \ref{GS} to the family $(\{\tilde{p_{2i}}\}_{0\leq i\leq 2m-1},p_{\infty})$. 
Following this process we define :

\begin{itemize}
\item For $0\leq i\leq 2m-1$, $i\neq m-2$, $i\neq m+1$ : $p_{2i}^N=\tilde{p}_{2i}\in{\CH_n(X\times_S X)}$
\item $p_{\infty}^N=p_{\infty}-1/2p_{\infty}\tilde{p}_{2m+2}-1/2\tilde{p}_{2m-4}p_{\infty}
=p_{\infty}-1/2p_{\infty}p_{2m+2}-1/2p_{2m-4}p_{\infty}\in{\CH_n(X\times_S X)}$
\item $p_{2m-4}^N=\tilde{p}_{2m-4}-1/2\tilde{p}_{2m-4}p_{\infty}=
\tilde{p}_{2m-4}-1/2p_{2m-4}p_{\infty}\in{\CH_n(X\times_S X)}$
\item $p_{2m+2}^N=\tilde{p}_{2m+2}-1/2p_{\infty}\tilde{p}_{2m+2}=
\tilde{p}_{2m+2}-1/2p_{\infty}p_{2m+2}\in{\CH_n(X\times_S X)}$.
\end{itemize}

The $p_{2i}^N$, $0\leq i\leq 2m-1$, and $p_{\infty}^N$ constitute then an orthogonal family of idempotents.

We now look at the action of these projectors on the higher direct images.





\begin{prop}
\label{act}
\begin{enumerate}[(i)]

\item For $0\leq i\leq 2m-1$ and $0\leq j\leq 2m-1$, $j\neq m$
$p_{2i,*}R^{2j}f_*{\mathbb Q}=\delta_{i,j}R^{2j}f_*\mathbb Q$

\item For $j=m$ 
$p_{2i,*}i^*R^{2m}g_*\mathbb Q=\delta_{i,m}i^*R^{2m}g_*\mathbb Q$

\item 
$p_{\infty,*}((R^{2m}f_*\mathbb Q)_v)=I_{(R^{2m}f_*\mathbb Q)_v)}$.

\end{enumerate}
\end{prop}

\begin{proof}

As for $j\neq m$ $R^{2j}f_*\mathbb Q=i^*R^{2j}g_*\mathbb Q={\mathbb Q}_S$ a constant sheaf,
it is enough to show that $p_{2i,*}(R^{2j}f_*\mathbb Q)_{|U})=\delta_{i,j}I_{(R^{2j}f_*Q)_{|U}}$ 
with $U=S\backslash C$ the open subset of $S$ over which $f$ is smooth.
We see this immediately by the action on the smooth fibers.
This proves (i).

The same technique proves (ii).

For (iii), let $V\subset{S}$ be an open subset(with $V\cap C\neq\emptyset$) and $\alpha\in{\Gamma(V,(R^{2m}f_*\mathbb Q)_v)}$.
By the Appendix \ref{app}(2), $\gamma_*$ induces an isomorphism of sheaves on $S$
$$\gamma_*:(\pi_*\mathbb Q)^-\tilde{\rightarrow}(R^{2m}f_*\mathbb Q)_v$$
So let $\beta\in{\Gamma(V,(\pi_*\mathbb Q)^-)}$ such that $\alpha=\gamma_*\beta$.
We have then by the Appendix \ref{app}(1)
$$q_{m,C,*}^+\alpha=\gamma\gamma^t_*\gamma_*\beta=\gamma_*d^2(I-\tau_*)\beta=\gamma d^2(2\beta)=2d^2\alpha.$$
We have in the same way
$$q_{m,C,*}^-=\gamma_*\tau_*d^2(I-\tau_*)\beta=\gamma_*d^2(I-\tau_*)\beta=\gamma_*d^2(-2\beta)=-2d^2\alpha.$$

\end{proof}

\begin{prop}
\label{acts}
\begin{enumerate}[(i)]
\item $p_{\infty,*}^N(R^{2m}f_*\mathbb Q)_v=I_{(R^{2m}f_*\mathbb Q)_v}$
\item $p_{2i,*}^Ni^*R^{2j}g^*\mathbb Q=\delta_{i,j}i^*R^{2j}g_*\mathbb Q$
\item $p_{2i,*}^N(R^{2m}f_*\mathbb Q)_v=0$
\item $p_{\infty,*}^Ni^*R^{2j}g^*\mathbb Q=0$.

\end{enumerate}
\end{prop}

\begin{proof}

We have $p_{\infty}^N=p_{\infty}+\omega\in{\CH_n(X\times_S X)}$
with $\omega=-1/2p_{\infty}p_{2m+2}-1/2p_{2m-4}p_{\infty}$.
Since $p_{2m-4}p_{\infty}p_{2m+2}=0$ for dimensional reasons and 
$$p_{2m-4}p_{\infty}p_{2m-4}=p_{2m+2}p_{\infty}p_{2m+2}=0$$ by Proposition \ref{ortho1}, we obtain
$\omega^2=0$. Hence, $\omega_*(R^{2m}f_*\mathbb Q)_v=0$ since it is a local system of rank one on $C$.

By Proposition \ref{act}, $\tilde{p_{2i}}=\delta_{ij}I_{i^*R^{2j}g_*\mathbb Q}$.
Thus, for $0\leq i\leq 2m-1$, $i\neq m-2$, $i\neq m+1$ :
$p_{2i,*}^Ni^*R^{2j}g^*\mathbb Q=\delta_{i,j}i^*R^{2j}g^*\mathbb Q$.

For $i=m-2$, $p_{2m-4}^N=\tilde{p}_{2m-4}+\eta$ with $\eta=-1/2p_{2m-4}p_{\infty}$.
Since $p_{\infty}$ is supported on $X_{C}\times_C X_{C}$,
$\eta$ acts as zero on the invariant part, and the result follows.

Parts (iii) and (iv) follow from parts (i) and (ii) and the orthogonality of the projectors $p_{2i}^N$ and $p_{\infty}^N$.

\end{proof}

\section{Relative Chow-Künneth decomposition II}

Let $f:X\rightarrow S$ be a quadric bundle of relative dimension $2m-1$, $X$ and $S$ smooth projective with $dim(S)=2$.
Then $dim(X)=n=2m+1$.
For $0\leq i\leq 2m-1$, let $\xi_i\subset{X}$ be a relative linear section of $f:X\hookrightarrow\mathbb P(\mathcal{E})\rightarrow S$
of relative dimension $i$.

The main theorem of this section is the following :

\begin{thm}
\label{actnilmain}

Let $R\in{\CH_n(X\times_S X)}$ be a relative correspondence. If $R$ acts as zero on $Rf_*\mathbb Q$, then
$R$ is nilpotent, more precisely $R^9=0$.

\end{thm}

Note $U=S\backslash C$ the open discriminant complement of $S$.
Let $j':U\hookrightarrow S$, $j:X_U\rightarrow X$, and
$j\times j:X_U\times_U X_U\hookrightarrow X\times_S X$ be the open immersions.

Denote $\iota:C\hookrightarrow S$, $i:X_{C}\hookrightarrow X$ and 
$i\times i:X_{C}\times_{C}X_{C}\hookrightarrow X\times_S X$ the closed immersion.

We can describe the Chow group $\CH_n(X\times_S X)$ via the localization exact sequence :
$$\CH_n(X_{C}\times_{C}X_{C})\stackrel{(i\times i)_*}{\rightarrow}\CH_n(X\times_S X)\stackrel{(j\times j)^*}{\rightarrow}\CH_n(X_U\times_U X_U)\rightarrow 0.$$

\begin{lemme}

Put $R_U=R_{|X_U\times_U X_U}$. Then $R_U^3=0$.

\end{lemme}

\begin{proof}

Note $\xi_{i,U}=\xi_i\cap X_U\subset{X}$. 
They are relative linear sections of  $f_U:X_U\hookrightarrow \mathbb P(\mathcal{E})_{|U}\rightarrow U$.
Since $f_U:X_U\rightarrow U$ is a smooth quadric bundle of relative dimension $2m-1$, 
 $$\phi_U : \CH_2(U)^{\oplus 2m}\oplus \CH_1(U)^{\oplus 2m-1}\oplus \CH_0(U)^{\oplus 2m-2}\rightarrow \CH_n(X_U\times_U X_U)$$
induced by $[\xi_{i,U}\times_U\xi_{2m-1-i,U}]$, $[\xi_{i,U}\times_U\xi_{2m-i,U}]$ and $[\xi_{i,U}\times_U\xi_{2m+1-i,U}]$ is surjective 
by the Appendix, Proposition \ref{OddQ}.

Since $\phi_U$ is surjective, there exists rational numbers $n_i$ such that 
$$R_U=R_{|X_U\times_U X_U}=\sum_{i=0}^{2m-1}n_ip_{2i,U}+\omega_U\in{\CH_n(X_U\times_U X_U)},$$
with $\omega_{U}\in{K_U=\phi_U(\CH_1(U)^{\oplus 2m-1}\oplus\CH_0(U)^{\oplus 2m-1})}$.

Let $V\subset{U}$ be a contractible open subset and for $0\leq j\leq 2m-1$ take \\
$0\neq\alpha\in{\Gamma(V,R^{2j}f_*\mathbb Q)=H^{2j}(X_V,\mathbb Q)}$.
We have then $R_{U,*}\alpha=\sum_{i=0}^{2m-1}n_ip_{2i,U,*}\alpha=n_j\alpha$ by Proposition \ref{act}.
Thus $n_j\alpha=0$. This gives $n_j=0$ for $0\leq j\leq 2m-1$. Hence
$R_U=\omega_U\in{\CH_n(X_U\times_U X_U)}$.
The Proposition \ref{NOdd} of the Appendix tells us then that $R_U^3=0$.

\end{proof}

By the compatibility of relative correspondences with flat base change $R^3_{|X_U\times_U X_U}=R_U^3=0$.
The localization exact sequence then says that
$$R^3=(i\times i)_*T\in{\CH_n(X\times_S X)},$$ 
with $T\in{\CH_n(X_{C}\times_{C}X_{C})}$.

Recall that $f_C:X_C\rightarrow C$ admits a section $e$. 
Let $\tilde{X}$ be the blow up of $X$ along $e(C)$ and let $\tilde{X_C}$ be the strict transform of $X_C$.
The morphism $\tilde{X_C}\rightarrow C$ factors as $\tilde{X_C}\stackrel{h}{\rightarrow}X^H\stackrel{f^H}{\rightarrow}C$,
where $h:\tilde{X_C}\to X^H$ is a $\mathbb P^1$ bundle and $f^H:X^H\to C$ is a smooth quadric bundle of relative dimension $2m-2$.
Note $\epsilon:\tilde{X_C}\to X_C$. Let $\xi_{j,C}^H\subset{X^H}$ be relative linear sections of $f^H$.

\begin{lemme}
\label{TH}
We have $T=(\epsilon\times\epsilon)_*(h\times h)^* T^H$ with $T^H\in{\CH_{n-2}(X^H\times_C X^H)}$.

\end{lemme}

\begin{proof}

Denote $$X_{C}^o=X_{C}\backslash e(C)$$ and 
$$(X_{C}\times_{C}X_{C})^o=(X_{C}\times_C X_{C})\backslash(X_{C}\times_{C}e(C)\cup e(C)\times_{C}X_{C}).$$
Let $p:X_{C}^o\rightarrow X^H$ be the relative projection from the section $e(C)$ on $H_{C}$
and
$p\times p:(X_C\times_C X_C)^o\rightarrow X^H\times_C X^H$.
Denote
$l:(X_C\times_C X_C)^o\hookrightarrow X_{C}\times_{C}X_{C}$ 
the open immersion.

Consider the commutative diagram : 

$$\xymatrix{\CH_{n-2}(X^H\times_C X^H) \ar[r]^{(p\times p)^*}\ar[d]^{(h\times h)^*}& \CH_n((X_C\times_C X_C)^o)\\
\CH_n(\tilde{X_C}\times_C\tilde{X_C}) \ar[r]^{(\epsilon\times\epsilon)_*}& \ar[u]^{l^*}\CH_n(X_C\times_C X_C)}.$$

Since $p\times p:(X\times_{C}X)^o\rightarrow X^H\times_C X^H$ is an ${\mathbb A}^2$ fibration,
$(p\times p)^*:\CH_{n-2}(X^H\times_{C}X^H)\rightarrow \CH_n(X_{C}\times_{C} X_{C})$ 
is an isomorphism.

Note that
$l^*:\CH_n(X_{C}\times_{C} X_{C})\rightarrow \CH_n((X_{C}\times_{C}X_{C})^o)$
is an isomorphism because $dim(X_{C}\times_{C}e(C)\cup e(C)\times_{C}X_{C})=n-1$.

We have thus $T=(l^*)^{-1}(p\times p)^*T^H=(\epsilon\times\epsilon)_*(h\times h)^*T^H$.

\end{proof}

Put $\lambda=i\circ\epsilon:\tilde{X_C}\rightarrow X$
and write $q_{2j,C}=(\lambda\times\lambda)_*(h\times h)^*p_{2j}(X^H)\in{\CH_n(X\times_S X)}$
Denote $pr:X\times_S X\rightarrow S$ the structural morphism.

\begin{lemme}
\label{Exp}
Let $\xi_j\subset{X}$ be a linear section that extends the linear section 
$\xi_{j,C}=\epsilon(h^{-1}(\xi_{j,C}^H)\subset{X_C}$.
We have
$$q_{2j,C}=[\xi_{j+1}\times_S\xi_{2m-1-j}].pr^*[C]\in{\CH_n(X\times_S X)}.$$

\end{lemme}

\begin{proof}

Since $\tilde{X_C}\rightarrow X^H$ is a $\mathbb P^1$-fibration we obtain
$(\epsilon\times\epsilon)_*(h\times h)^*[\xi_{j,C}\times_C\xi_{2m-2-j,C}]=(i\times i)^*[\xi_{j+1}\times_S\xi_{2m-1-j}]$.
The result then follows from the projection formula since $(i\times i)_*[X_C\times_C X_C]=pr^*[C]$.

\end{proof}

\begin{rem}

Note that
$$(\lambda\times\lambda)_*(h\times h)^*p_{2m-2}^{+}(X^H)=\lambda_*h^*\Gamma_W^Hg^*g_*(\Gamma_W^H)^th_*\lambda^*
=\gamma\gamma^t=q_{m,C}^+$$ and
$$(\lambda\times\lambda)_*(h\times h)^*p_{2m-2}^{-}(X^H)=\lambda_*h^*\Gamma_W^Hg^*\tau_*g_*(\Gamma_W^H)^th_*\lambda^*
=\gamma\tau^*\gamma^t=q_{m,C}^-.$$

\end{rem}

\begin{lemme}
\label{subspacenil}
Let $K^H_C\subset{\CH_{n-2}(X^H\times_C X^H)}$ be the subspace \ref{Sub} defined in the Appendix.
Put $K_C=(\lambda\times\lambda)_*(h\times h)^*K^H_C$ and 
let $L=\sum_{j=0}^{2m-2}n_jq_{2j,C}+\omega\in{\CH_n(X\times_S X)}$ with $\omega\in{K_C}$.
Then $L^3=0$.

\end{lemme}

\begin{proof}
It suffices to shows that,
for $\omega\in{K_{C}}$ and $\eta\in{K_C}$,
\begin{enumerate}
\item [a)] $q_{2i,C}q_{2j,C}q_{2k,C}=0$.
\item [b)] $q_{2j,C}\omega=\omega q_{2j,C}=0$ 
\item [c)] $\omega\eta=0$.
\end{enumerate}

We start by proving a). By the Lemma \ref{Exp}, we have
$q_{2j,C}=[\xi_{2m-1-j}]\circ[C]\circ[\xi_{j+1}]^t$. Hence,
$$q_{2i,C}q_{2j,C}q_{2k,C}=[\xi_{2m-1-i}][C][\xi_{i+1}]^t[\xi_{2m-1-j}][C][\xi_{j+1}]^t[\xi_{2m-1-k}][C][\xi_{k+1}]^t.$$
We claim that 
\begin{equation}
\label{formula}
[C][\xi_{i+1}]^t[\xi_{2m-1-j}][C][\xi_{j+1}]^t[\xi_{2m-1-k}][C]=0.
\end{equation}
To prove the claim,
consider the expressions $[\xi_{i+1}]^t[\xi_{2m-1-j}]\in{\Corr^p_S(S,S)=\CH^p(S)}$ and $[\xi_{j+1}]^t[\xi_{2m-1-k}]\in{\Corr^q_S(S,S)=\CH^q(S)}$. 
If $p<0$ or $q<0$ the formula holds. If not then the expression (\ref{formula}) belongs to $\CH^{p+q+3}(S)$ and vanishes since $p+q\geq 0$.

The subspace $K_C\in{\CH_n(X\times_S X)}$ is generated by the elements 
\begin{eqnarray*}
q_{j,s'}&=&(\lambda\times\lambda)_*(h\times h)^*(\rho\times\rho)_*([\xi^H_{j-1,s'}\times\xi^H_{2m-j,s'}]) \, \mbox{for} \,
 2\leq j\leq 2m-1 \, \, j\ne m \, \, j\ne m+1 \\
q_{m+1,s'}&=&(\lambda\times\lambda)_*(h\times h)^*(\rho\times\rho)_*([\xi^H_{m,s'}\times\Lambda'^H_{s'}]) \\
q_{m,s'}&=&(\lambda\times\lambda)_*(h\times h)^*(\rho\times\rho)_*([\Lambda'^H_{s'}\times\xi^H_{m,s'}]).
\end{eqnarray*}
where $s'\in{C'}$.

Note that $q_{j,s'}=q_{2m+1-j,s'}^t$ for $s'\in{C'}$ and $2\leq j\leq 2m-1$.

We have
\begin{eqnarray*}
q_{j,s'}&=&\lambda_*h^*\rho_*[\xi^H_{2m-j,C'}][s'][\xi^H_{j-1}]^t\rho^*h_*\lambda^* \, \mbox{for} \, 2\leq j\leq 2m-1 \, \, j\ne m \, \, j\ne m+1  \\
q_{m+1,s'}&=&\lambda_*h^*\rho_*[\Lambda'^H][s'][\xi^H_{m,C'}]^t\rho^*h_*\lambda^* \\
q_{m,s'}&=&\lambda_*h^*\rho_*[\xi^H_{m,C'}][s'][\Lambda'^H]^t\rho^*h_*\lambda^*.
\end{eqnarray*}

For c), let us prove for instance that $q_{m,s'}q_{m,t'}=0$, the other equalities are similar. We claim that
\begin{equation}
\label{formule}
[s'][\xi_{m+1,C'}^H]^t\rho^*h_*\lambda^*\lambda_*h^*\rho_*[\Lambda'^H][t']=0. 
\end{equation}
Consider
the expression $[\xi_{m+1,C'}^H]^t\rho^*h_*\lambda^*\lambda_*h^*\rho_*[\Lambda'^H]\in{\Corr^p_S(C',C')=\CH^p(C'\times_C C')}$. If $p<0$ the formula holds.
If not the expression \ref{formule} belongs to $\CH^{2+p}(C'\times_C C')$ and vanishes since $p\geq 0$.

Let us prove b).
We have $q_{m,s'}q_{2j,C}=0$ for $0\leq j\leq 2m-2$. As before, it follows since 
$[s'][\Lambda'^H]^t\rho^*h_*\lambda^*[\xi_{2m-1-j}][C]=0$. Indeed
$[\Lambda'^H]^t\rho^*h_*\lambda^*[\xi_{2m-1-j}]\in{\Corr^{p}_S(S,C')=\CH^{p}(C')}$ vanishes if $p<0$
and $[s'][\Lambda'^H]^t\rho^*h_*\lambda^*[\xi_{2m-1-j}][C]\in{\Corr^{p+2}_S(S,C')=\CH^{p+2}(C')}$.
Similarly $q_{2j,C}q_{m,s'}=0$ for $0\leq j\leq 2m-2$ since $[C][\xi_{j+1}]^t\lambda_*h^*\rho_*[\xi^H_{m,C'}][s']=0$.
Indeed $[\xi_{j+1}]^t\lambda_*h^*\rho_*[\xi^H_{m,C'}]\in{\Corr^p_S(C',S)=\CH^{p-1}(C')}$ vanishes if $p<1$
and $[C][\xi_{j+1}]^t\lambda_*h^*\rho_*[\xi^H_{m,C'}][s']\in{\Corr^{p+2}_S(C',S)=\CH^{p+1}(C')}$.
By transposition we see that similarly $q_{m+1,s'}q_{2j,C}=q_{2j,C}q_{m+1,s'}=0$ for $0\leq j\leq 2m-2$.
Similarly $q_{2j,C}q_{k,s'}=0$ for $2\leq k\leq 2m-1$, $k\ne m$, $k\ne m+1$ and $0\leq j\leq 2m-2$ 
since $[C][\xi_{j+1}]^t\lambda_*h^*\rho_*[\xi^H_{2m-k,C'}][s']=0$.
By transposition we see that similarly $q_{k,s'}q_{2j,C}=0$ for $2\leq k\leq 2m-1$, $k\ne m$, $k\ne m+1$ and $0\leq j\leq 2m-2$.

\end{proof}

\begin{lemme}
\label{coef}
We have
$T^H=\sum_i n_i p_{2i}(X^H)+\omega^H$ with $\omega^H\in{K^H_C}$, where $T^H$ is the cycle introduced in Lemma \ref{TH}.

\end{lemme}

\begin{proof}

By the Corollary \ref{EvenQ} and Lemma \ref{Gen} of the Appendix applied to the quadric bundle 
$f^H_{C}:X^H_{C}\rightarrow C$, there exist rational numbers $n_j$ such that
$$T^H=\sum_{j=0,j\neq m-1}^{2m-2}n_jp_{2j}(X^H)+n^+p_{2m-2}^+(X^H)+n^-p_{2m-2}^-(X^H)+\omega^H
\in{\CH_{n-2}(X^H\times_{C} X^H)},$$
with $\omega^H\in{K^H_{C}}$.

We obtain
\begin{equation*} 
R^3=(\lambda\times\lambda)_*(h\times h)^*T^H=\sum_{j=1,j\neq m}^{2m-1}n_jq_{j,C}+n^+q_{m,C}^++n^-q_{m,C}^-+\omega\in{\CH_n(X\times_S X)},
\end{equation*}
with $\omega\in{K_{C}}$.

Let $V\subset{S}$ be a contractible open subset of $S$ such that $C\cap V\neq\emptyset$ and
take $0\neq\alpha\in{H^0(V,(R^{2m}f_*\mathbb Q)_v)=H^{2m}(X_V,\mathbb Q)_v}$.
As $q_{j,C}\in{\CH_n(X\times_S X)}$ and $\omega\in{\CH_n(X\times_S X)}$ are nilpotent by the Lemma \ref{subspacenil},
$q_{j,C*}\alpha=0$ and $\omega_*\alpha=0$. Moreover the proof of Proposition 2.7(iii) shows that
$$R_*\alpha=(\sum_{j=1,j\neq m}^{2m-1}n_jq_{j,C}+n^+q_{m,C}^++n^-q_{m,C}^-+\omega)_*\alpha
=2d^2(n^+-n^-)\alpha.$$
As $\alpha\neq 0$, we obtain $n^+=n^-=n$.

Hence,
\begin{equation*} 
T^H=\sum_{j=0,j\neq m-1}^{2m-1}n_jp_{2j}(X^H)+n(p_{2m-2}(X^H)^++p_{2m-2}^-)+\omega^H\in{\CH_n(X^H\times_C X^H)}.
\end{equation*}

Note that $p_{2m-2}^+(X^H)+p_{2m-2}^-(X^H)=d^2p_{2m-2}(X^H)+\omega_2^H$, with $\omega_2^H\in{K_C^H}$
by Lemma \ref{actnileven}.

We obtain finally :

\begin{eqnarray*} 
T^H&=&\sum_{j=0,j\neq m-1}^{2m-2}n_jp_{2j}(X^H)+n(d^2p_{2m-2}(X^H)+\omega^H_2)+\omega^H\\
&=&\sum_{j=0}^{2m-2}n_jp_{2j}(X^H)+\omega^H_3\in{\CH_n(X^H\times_C X^H)},
\end{eqnarray*}
with $n_{m}=nd^2$ and
$\omega^H_3=\omega^H+n\omega^H_2\in{K^H_{C}}$.

\end{proof}

We can now prove Theorem \ref{actnilmain}.
\begin{proof}(Theorem \ref{actnilmain})
We have by Lemma \ref{coef}
\begin{equation} 
R^3=\sum_{j=0}^{2m-2}n_jq_{2j,C}+\omega_3\in{\CH_n(X\times_S X)},
\end{equation}
with $\omega_3\in{K_{C}}$.
The Lemma \ref{subspacenil} then shows that $R^9=(R^3)^3=0$. This completes the proof of Theorem \ref{actnilmain}.
\end{proof}

\begin{proposition}
\label{motives}
\begin{enumerate}
\item The mutually orthogonal projectors $p_{2i}^N\in{\CH_n(X\times_S X)}$, $0\leq i\leq 2m-1$, and $p_{\infty}^N\in{\CH_n(X\times_S X)}$ constitute a relative Chow-K\"unneth decomposition of $f:X\rightarrow S$.
\item
We have the following isomorphisms of relative Chow motives :
\begin{enumerate}
\item (1)$(S,\Delta_S)(i)\tilde{\rightarrow}(X,p_{2i})\tilde{\rightarrow}(X,p_{2i}^N)$ for $0\leq i\leq 2m-1$.
\item (2) $(\tilde{C},\rho)(m)\tilde{\rightarrow}(X,p_{\infty})\tilde{\rightarrow}(X,p_{\infty}^N)$.
\end{enumerate}
\end{enumerate}
\end{proposition}

\begin{proof}

For the proof of part one we consider now
$$R=\Delta_X-\sum_{i=0}^{2m-1}p_{2i}^N-p_{\infty}^N\in{\CH_n(X\times_S X)}.$$
As $R$ acts by zero on $Rf_*\mathbb Q$ by Proposition \ref{acts}, Theorem \ref{actnilmain} tells us that $R^9=0$. 
\\$R$ being an idempotent : $R=R^2=R^9$, so $R=0$ i.e. 
$$\Delta_X=\sum_{i=0}^{2m-1}p_{2i}^N+p_{\infty}^N\in{\CH_n(X\times_S X)}.$$

Let us prove part two.
The first isomorphism of (1) is induced by $[\xi_i]\in{\Corr^{2m-1-i}_S(S,X)}$ and $[\xi_{2m-1-i}]^t\in{\Corr^{-(2m-1-i)}_S(S,X)}$
is the inverse.
The second isomorphism of (1) is induced by $p_{2i}p_{2i}^N\in{\Corr^0_S(X,X)}$ and $p_{2i}^Np_{2i}\in{\Corr^0_S(X,X)}$ is the inverse
(note that $p_{2i}^N-p_{2i}$ is nilpotent of order $2$).
The first isomorphism of (2) is induced by $\gamma\in{\Corr^{m}_S(\tilde{C},X)}$ and $\gamma^t\in{\Corr^{-m}_S(X,\tilde{C})}$
is the inverse.
The second isomorphism  of (2) is induced by $p_{\infty}p_{\infty}^N\in{\Corr^0_S(X,X)}$ and $p_{\infty}^Np_{\infty}\in{\Corr^0_S(X,X)}$
is the inverse(note that $p_{\infty}^N-p_{\infty}$ is nilpotent of order $2$).

\end{proof}

Combining parts 1 and 2 of Proposition \ref{motives} we obtain the following :

\begin{corollaire}
\label{motif}
We have an isomorphism of relative Chow motives :
$$(X,\Delta_X)\tilde{\rightarrow}\oplus_{i=0}^{2m-1}(S,\Delta_S)(-i)\oplus(\tilde{C},\rho)(-m).$$

\end{corollaire}






\section{Consequences}

\subsection{Absolute Chow-K\"unneth decomposition}

Let $f:X\rightarrow S$ a quadric bundle with odd dimensional fibers $2m-1$, $X$, $S$ smooth projective, $dim(S)=2$ and $C\subset{S}$ smooth.

Let for $0\leq i\leq 2m-1$, $\xi_i\subset{X}$ be relative linear sections of $X\hookrightarrow\mathbb P(\mathcal{E})\rightarrow S$.

Denote $k:X\times_S X\hookrightarrow X\times X$ the closed embedding.
\\ $k_*:\CH_n(X\times_S X)\rightarrow \CH_n(X\times X)$ gives the functor from relative to absolute correspondences\cite{livre motif}.

$S$ being a smooth projective surface, it admits a Chow-K\"unneth decomposition given by projectors $p_j(S)\in{\CH_2(S\times S)}$
for $0\leq j\leq 4$ which satisfy Murre's conjectures\cite{livre motif}

Let $\pi_{2i}$ be the image under this functor of the projector $p_{2i}^N(X/S)$ and let $\pi_{\infty}$ be the image of $p_{\infty}^N$.
By Proposition \ref{motives} (2) and Murre's result we have $(X,\pi_{2i})\simeq(S,\Delta_S)(-i)\simeq\oplus\sum_{j=0}^4(S,p_j(S))(-i)$.
This means that we have  $\pi_{2i}=\sum_{j=0}^4 \pi_{2i,j}$. 
Put $$
p_k(X)=\left\{
\begin{array}{cc}
\sum_{2i+j=k}\pi_{2i,j}& k\ne 2m+1\\
\sum_{2i+j=k}\pi_{2i,j}+\pi_{\infty}&k=2m+1.
\end{array}
\right.$$

\begin{thm}
If the double covering $\tilde{C}\to C$ is non trivial,the projectors $p_k(X)$ induces an absolute Chow-Künneth decomposition of $X$ and 
\begin{equation}
\label{absmotive}
(X,\Delta_X)\tilde{\rightarrow}\oplus_{i=0}^{2m-1}(S,\Delta_S)(-i)\oplus(\tilde{C},\rho)(-m).
\end{equation}
\end{thm}

\begin{proof}

By construction the projectors $p_k(X)$ are mutually orthogonal and add up to the diagonal $\Delta_X$.
To show that it gives an absolute decomposition we have to show that the action of $p_k(X)$ on $H^l(X)$ is zero for
$k\neq l$. Since the action of $\pi_{2i,k-2i}$ on $H^l(X)$ factors through the action of $p_{k-2i}(S)$ on $H^{l-2i}(S)$ the result
is clear for $k\ne 2m+1$. For the remaining case put $M_{\infty}=(X,\pi_{\infty})$ and note that if the double covering $\tilde{C}\to C$ is non trivial
we have an isomorphism $$H^i(M_{\infty})\simeq H^{i-2m}(C,\mathbb L)$$ where $\mathbb L=(\pi_*\mathbb Q)^-$ is a non trivial rank one local system.
The result then follows since $H^j(C,\mathbb L)=0$ for $j\ne 1$.
The formula (\ref{absmotive}) follows from Corollary \ref{motif}.

\end{proof}

\begin{rem}
\label{trivcov}
If the double covering $\tilde{C}\to C$ is trivial, then the motive $(X,\pi_{\infty})$ is isomorphic to $(C,\Delta_C)(-m)$ and we have a decomposition
$(X,\Delta_X)\tilde{\rightarrow}\oplus_{i=0}^{2m-1}(S,\Delta_S)(-i)\oplus(C,\Delta)(-m)$.
\end{rem}

\begin{corollaire}

If the double covering $\tilde{C}\to C$ is non trivial, we have isomorphisms of $\mathbb Q$ Hodge structures
\begin{itemize}
\item $H^{2j}(X){\rightarrow}H^0(S)(-j)\oplus H^2(S)(j-1)\oplus H^4(S)(j-2)$
\item $H^{2j+1}(X)\tilde{\rightarrow}H^1(S)(-j)\oplus H^3(S)(j-1)$ if $j\neq m$
\item $H^n(X)\tilde{\rightarrow}H^1(S)(-m-1)\oplus H^3(S)(-m)\oplus H^1(\tilde{C})^-$
\end{itemize}
and isomorphisms of Chow groups with rational coefficients
\begin{itemize}
\item $\CH^0(S)\oplus \CH^1(S)(\oplus \CH^2(S)(-j+2)\tilde{\rightarrow}\CH^j(X)$ if $j\neq m+1$
\item $\CH^0(S)(-m-1)\oplus \CH^1(S)(-m)\oplus \CH^2(S)(-m+1)\oplus \CH^1(\tilde{C})(-m)\tilde{\rightarrow}\CH^{m+1}(X)$ 
\item $\CH^0(S)_{alg}\oplus \CH^1(S)_{alg}\oplus \CH^2(S)_{alg}\oplus \CH^1(\tilde{C})^-\tilde{\rightarrow}\CH^{m+1}(X)_{alg}$ 
\end{itemize}

\end{corollaire}

\begin{rem}
If $S=\mathbb P^2$ the previous result gives a motivic proof(with $\mathbb Q$-coefficients) of a theorem of Beauville \cite{A.Beauville}.
\end{rem}

\begin{proposition}
The projectors $p_k(X)$, $0\leq k\leq 2n$ of $X$ satisfy Murre's conjectures I, II and III. 
\end{proposition}

\begin{proof}

We have already verified conjecture I in Theorem \ref{absmotive}. To prove conjecture II, we have to show that $p_k(X)$ acts as zero on
$\CH^l(X)$ if $k\notin\left\{l,...,2l\right\}$.
Let us first consider the case $k\ne 2m+1$. Write $p_k(X)=\sum_i\pi_{2i,k-2i}$. Since the action of $\pi_{2i,k-2i}$ on $\CH^l(X)$
factors through the action of $p_{k-2i}(S)$ on $\CH^{l-i}(S)$, the action is zero if $k-2i\notin\left\{l-i,...,2l-2i\right\}$ by Murre's 
theorem. Hence the action of $p_k(X)$ on $\CH^l(X)$ is zero if $k\notin\left\{l+i,...,2l\right\}$(which is stronger then the statement that we need).
For $k=2m+1$ we write $p_{2m+1}(X)=\sum_i \pi_{i,2m+1-i}+\pi_{\infty}$. For the projectors $\pi_{i,2m+1-i}$ we use the same reasoning as above,
and the projector $\pi_{\infty}$ only acts on $\CH_{m+1}(X)$ since the Prym projector $\rho$ only acts on $\CH^1(\tilde{C})$ if the double covering 
is non trivial. If the double covering is trivial, we decompose the motive $\pi_{\infty}$ as in Remark \ref{trivcov}.

\end{proof}

\begin{rem}
As the motive of $X$ is a direct sum of a number of Tate twists of copies of the motive of the surface $S$ and a direct summand of the motive of a curve,
the motive $(X,\Delta_X)$ is finite dimensional if the motive of $S$ is finite dimensional. 
\end{rem}

\section{Appendix}

\subsection{Computation of $\gamma^t\circ\gamma$ }

Let $f:X\rightarrow S$ be a quadric bundle over a smooth projective surface $S$, with fibers of dimension $2m-1$. 
Let $W\subset{F_m(X/S)}$ a multisection of degree $d$ of $F_m(X/S)\rightarrow\tilde{C}$. 
Denote $i:C\hookrightarrow S$ the closed immersion 
and $\pi:\tilde{C}\rightarrow S$ be the double covering $\tilde{C}\rightarrow C$
composed with the immersion $i$. 
Recall that $$\gamma=\Gamma_W g^*\in{\Corr_C(\tilde{C},X)}.$$ 
The aim of this section is to compute 
$\gamma^t\circ\gamma\in{\Corr^0_C(\tilde{C},\tilde{C})}=\CH_1(\tilde{C}\times_C\tilde{C})$. We shall only sketch the proof. 
For details see \cite{NS} or the author's PhD thesis.

Note that 
$$\CH_1(\tilde{C}\times_C\tilde{C})\cong H_2^{BM}(\tilde{C}\times_C\tilde{C})\cong\End(\pi_*\mathbb Q),$$
hence it suffices to study the action of $\gamma^t\circ\gamma$ on $\pi_*\mathbb Q$.
To do this we choose a point $s\in{C}$ and a transversal slice $T$ to $C$ at $s$.
Put $X_T=f^{-1}(T)$ and let $\bar{X_T}$ be a smooth compactification of $X_T$ such that
$f_T:X_T\rightarrow T$ extends to $\bar{f}:\bar{X_T}\rightarrow\bar{T}$.

Restricting to $T$ we obtain
$$(\pi_*\mathbb Q )_s\stackrel{\gamma_T}{\rightarrow}R^{2m}f_{T*}\mathbb Q\stackrel{\gamma^t_T}{\rightarrow}(\pi_*\mathbb Q)_s$$

Put $\pi^{-1}(s)=\left\{s',s''\right\}$.
Since $(\pi_*\mathbb Q)_s=\mathbb Q[s']\oplus\mathbb Q[s'']$, the above map is given by a $2$ by $2$ matrix $A$ with rational coefficients.

\begin{proposition}
We have $$A=(-1)^m d^2
\begin{pmatrix}
1&-1\\
-1&1
\end{pmatrix}.$$
\end{proposition}

\begin{proof}

Given $t,u\in{\left\{s',s''\right\}}$, write $g^{-1}(t)=\left\{t_1,...,t_d\right\}$ and
write $g^{-1}(u)=\left\{u_1,...,u_d\right\}$.
The $(t,u)$-component of $(\gamma^t\circ\gamma)_{T}$ is given by the composition
$$H^0(\left\{t\right\})\rightarrow \oplus_{i=1}^dH^0(\Lambda_{t_i})\rightarrow H^{2m}_c(X_T)\rightarrow H^{2m}(X_T)\rightarrow\oplus_{j=1}^d H^{2m}(\Lambda_{u_j})\rightarrow H^0(\left\{u\right\}).$$
The map $H^{2m}_c(X_T)\rightarrow H^{2m}(X_T)$ factors through $H^{2m}(\bar{X_T})$ and the map
$$H^0(\Lambda_{t_i})\rightarrow H^{2m}_c(X_T)\rightarrow H^{2m}(X_T)\rightarrow H^{2m}(\Lambda_{u_j})$$ is given
by the intersection $(\Lambda_{t_i}.\Lambda_{u_j})$ in $\bar{X_T}$.

To calculate this intersection number,
note that the singular quadric $X_s=f^{-1}(s)$ is a cone with vertex $x$ over a smooth quadric $X^H_s$ of dimension $2m-2$ and the
two families of $m$-planes of $X_s$ are obtained by taking the cone over the two families of $(m-1)$-planes of $X^H_s$.
Also recall that if $\Lambda_1$ and $\Lambda_2$ belong to the same family of $m-1$ planes on $X^H_s$ we have
$$dim(\Lambda_1\cap\Lambda_2)\equiv m-1[2].$$
Suppose that $m$ is even. Then if $t=u$, by moving $\Lambda_{t_i}$ and $\Lambda_{t_j}$ in their rational equivalence class
we may assume that the intersection is the vertex $x$. Hence $\Lambda_{t_i}.\Lambda_{t_j}=1$.
If $u=\tau(t)$, we compute the intersection $\Lambda_{t_i}.\Lambda_{u_j}$ as follows.
Since $\Lambda_{t_i}+\Lambda_{u_j}$ is rationally equivalent in $\bar{X_T}$ to the intersection of $X_s$ with a relative linear subspace and
any two fibers are numerically equivalent we obtain $\Lambda_{t_i}.(\Lambda_{t_i}+\Lambda_{u_j})=0$.
Hence $\Lambda_{t_i}.\Lambda_{u_j}=-1$.
The proof for $m$ odd is similar.

\end{proof}

\begin{corollaire}
\label{app}
We have
\begin{enumerate}
\item [(1)]$\gamma^t\circ\gamma=(-1)^md^2(I-\tau_*)\in{\CH_1(\tilde{C}\times_C\tilde{C})}$
\item [(2)]$\gamma$ induces an isomorphism $\gamma_* (\pi_*\mathbb Q)^-\tilde{\rightarrow} (R^{2m}f_*\mathbb Q)_v$. 
\end{enumerate}

\end{corollaire}

\begin{proof}

Part one is the expression of the matrix $A$. Part two follows from the fact
that the composite morphism of sheaves on $S$
$(\pi_*\mathbb Q)^-\stackrel{\gamma_*}{\rightarrow}(R^{2m}f_*\mathbb Q)_v\stackrel{\gamma^t_*}{\rightarrow}(\pi_*\mathbb Q)^-$
is an isomorphism since if $s\in{C}$, $\pi^{-1}(s)=\left\{s',s''\right\})$ and it sends $[s']-[s'']\in{H^0(\pi^{-1}(s))}$ to 
$2(-1)^md^2([s']-[s''])\in{H^0(\pi^{-1}(s))}$ by the computation of the matrix $A$. 
Hence, $\gamma_*(\pi_*\mathbb Q)^-\to(R^{2m}f_*\mathbb Q)_v$ is injective.
Since $(R^{2m}f_*\mathbb Q)_v$ is a local system of rank one on $C$ and also $(\pi_*\mathbb Q)^-$, it is an isomorphism.

\end{proof}

\subsection{Chow groups of smooth quadric bundle}

\begin{df}
Let $X$ and $S$ be smooth irreducible algebraic varieties and $f:X\rightarrow S$ a smooth morphism.
We say that $f:X\rightarrow S$ admits a relative cellular decomposition in the weak sense if there exists a stratification by closed subvarieties
$X_{\alpha}\subset{X_{\alpha+1}}\subset...\subset{X}$ so that 
$$f_{\alpha}=f_{|X_{\alpha}\backslash X_{\alpha-1}}:X_{\alpha}\backslash X_{\alpha-1}=X_{\alpha,1}\sqcup ...\sqcup X_{\alpha,l}\rightarrow S$$ where $X_{\alpha,1}\rightarrow S$,\ldots ,$X_{\alpha,l}\rightarrow S$ are 
${\mathbb A}^k$ fibrations in the weak sense(not necessarily locally trivial).

\end{df}

\begin{prop}
\label{Cel}

If $f:X\rightarrow S$ admits a relative cellular decomposition $X=X_0\supset X_1\supset ...\supset X_\alpha\supset ...$ 
with $codim(X,X_\alpha)=c_{\alpha}$
\begin{enumerate}
\item [(i)]$\oplus_{\alpha}\CH^{k-c_{\alpha}}(S)\tilde{\rightarrow}\CH^k(X)$
\item [(ii)] $\oplus_{\alpha,\beta}\CH^{k-c_{\alpha}-c_{\beta}}(S)\tilde{\rightarrow}\CH^k(X\times_S X)$.
\end{enumerate}
\end{prop}

\begin{proof}

Part (i) has been proved by Köck, \cite{Köck} in the case where $X_{\alpha,i}\rightarrow S$ are locally trivial fibrations.
The result also holds for fibration in the weak sense \cite{CGM}.

Part (ii) follows in the same way since $X\times_S X\rightarrow S$ admits a relative cellular decomposition, \cite{Karpenko}.

\end{proof}

\begin{proposition}
Let $f:X\rightarrow S$ be a smooth quadric bundle.
There exists of finite covering $\pi:S'\rightarrow S$ so that the quadric bundle $f':X'=X\times_S S'\rightarrow S'$ 
admits a relative cellular decomposition.
\end{proposition}

\begin{proof}

Recall that in the case of a quadric $Q$ defined over $\mathbb C$ the cellular decomposition is obtained as follows. Chose a point $p\in Q$ and consider the intersection $Q\cap T_p Q$, which is a cone over a smooth quadric $Q'$ of dimension $dim(Q)-2$. By induction(starting with the case $\mathbb P^1$ and $\mathbb P^1\times\mathbb P^1$) $Q'$ admits a cellular decomposition. Hence $Q$ admits a cellular decomposition. 

We can apply the same process in the relative case if all the data needed to produce the cellular decomposition are defined over the base(if there exists a section of $F_i(X/S)\rightarrow S$ for $i\leq dim(X/S)/2$). This can be achieved by passing through a finite covering of the base.

\end{proof}

\begin{proposition}
\label{OddQ}

Let $f:X\rightarrow S$ a quadric bundle which is smooth of relative dimension $2m-1$. Assume for simplicity that $S$ is a surface.
Let $\xi_i\subset{X}$ be a relative linear section.  We have a surjection of $\mathbb Q$ vector spaces
$$\phi:\CH_2(S)^{\oplus 2m}\oplus CH_1(S)^{\oplus 2m-1}\oplus\CH_0(S)^{\oplus 2m-2}\rightarrow \CH_n(X\times_S X)$$
defined by
\begin{eqnarray*}\phi(\alpha_0,...,\alpha_{2m-1},\beta_1,...,\beta_{2m-1},\gamma_2,...,\gamma_{2m-1})=
([\xi_{2m-1}\times_S\xi_0]_*(\alpha_0),...,[\xi_{0}\times_S\xi_{2m-1}]_*(\alpha_{2m-1})\\ ,[\xi_{2m-1}\times_S\xi_1]_*(\beta_1),...,[\xi_{1}\times_S\xi_{2m-1}]_*(\beta_{2m-1}),
[\xi_{2m-1}\times_S\xi_2]_*(\gamma_2),...,[\xi_2\times_S\xi_{2m-1}]_*(\gamma_{2m-1})
\end{eqnarray*}

\end{proposition}

\begin{proof}

By Proposition \ref{Cel} we can take a finite covering $\pi:S'\rightarrow S$ such that $f':X'=X\times_S S'\rightarrow S'$ admits a relative cellular
decomposition. Denote $\rho:X'\rightarrow X$ and $\rho:X\times_{S'}X'\rightarrow X\times_S X$ the finite coverings obtained by base change.
Consider the relative linear sections $\xi'_i=\pi'^{-1}(\xi_i)\subset{X'}$. We have the following commutative diagram :

$$
\xymatrix{ \CH_2(S')^{\oplus 2m}\oplus \CH_1(S')^{\oplus 2m-1}\oplus\CH_0(S')^{\oplus 2m-2} \ar[r]^{\pi_*} \ar[d]^{\phi'}& \CH_2(S)^{\oplus 2m}\oplus \CH_1(S)\oplus  \CH_0(S)^{\oplus 2m-r} \ar[d]^{\phi}\\
\CH_n(X'\times_{S'} X') \ar[r]^{(\rho\times\rho)_*}&\CH_n(X\times_S X)}.
$$

Indeed since we took $\xi'_i=\pi'^{-1}(\xi_i)$ we have $[\xi'_{i}\times_{S'}\xi'_{2m+1-i-k}]=\rho^*[\xi_{i}\times_S\xi_{2m+1-i-k}]$.
So by the projection formula,
for $\alpha\in{\CH_k(X)}$ we have 
$$[\xi_{i}\times_{S}\xi_{2m+1-i-k}]_*\alpha=(\rho\times\rho)_*([\xi'_{i}\times_{S'}\xi'_{2m+1-i-k}]_*\rho_*\alpha).$$

As $\pi:S'\rightarrow S$ and $\rho\times\rho:X'\times_{S'} X'\rightarrow X\times_S X$ are surjective (they are finite covering),
$\pi_*:\CH_{k'}(S)\rightarrow\CH_k(S)$ and $(\rho\times\rho)_*:\CH_n(X'\times_{S'} X')\rightarrow\CH_n(X\times_S X)$ are surjective.
As the quadric bundle $f':X'\rightarrow S'$ admits a relative cellular decomposition, (i) says that 
$\phi'$ is an isomorphism.
Thus by the diagram chase $\phi$ is surjective.

\end{proof}

We immediately deduce the following :

\begin{proposition}
\label{NOdd}

Let $f:X\rightarrow S$ be a quadric bundle which is smooth of relative dimension $2m-1$.
Assume for simplicity that $S$ is a surface.
Denote $$K_S=\phi(\CH_1(S)^{\oplus 2m-1}\oplus\CH_0(S)^{\oplus 2m-2}).$$. 
Let $F\in{CH_n(X\times_S X)}$ be a relative correspondence of degree zero.
If $F\in{K_S}$, then $F^3=0$.

\end{proposition}

\begin{proof}

Denote for $C\subset{S}$ an irreducible curve and $1\leq j\leq 2m-1$ 
$$q_{j,C}=[\xi_{j}\times_S\xi_{2m-j}]_*[C]=[\xi_{j}\times_S\xi_{2m-j}].pr^*[C]=[\xi_{j,C}\times_C\xi_{2m-j,C}]\in{CH_n(X\times_S X)}$$
We have $q_{j,C}=[\xi_{2m-j}]\circ[C]\circ[\xi_{j}]^t$. 
Let us show that $q_{i,C_1}q_{j,C_2}q_{k,C_3}=0$.
We have
$$q_{i,C_1}q_{j,C_2}q_{k,C_3}=[\xi_{2m-i}][C_1][\xi_{i}]^t[\xi_{2m-j}][C_2][\xi_{j}]^t[\xi_{2m-k}][C_3][\xi_{k}]^t$$
We claim that 
\begin{equation}
\label{formula}
[C_1][\xi_{i+1}]^t[\xi_{2m-1-j}][C_2][\xi_{j+1}]^t[\xi_{2m-1-k}][C_3]=0.
\end{equation}
To prove the claim,
consider the expressions $[\xi_{i+1}]^t[\xi_{2m-1-j}]\in{\Corr^p_S(S,S)=\CH^p(S)}$ and $[\xi_{j+1}]^t[\xi_{2m-1-k}]\in{\Corr^q_S(S,S)=\CH^q(S)}$. 
If $p<0$ or $q<0$ the formula holds. If not then the expression (\ref{formula}) belongs to $\CH^{p+q+3}(S)$ and vanishes since $p+q\geq 0$.

Denote for $s\in{S}$ and $2\leq j\leq 2m-1$
$$q_{j,s}=[\xi_{j}\times_S\xi_{2m+1-j}]_*[s]=[\xi_{j}\times_S\xi_{2m+1-j}].pr^*[s]=[\xi_{j,s}\times\xi_{2m+1-j,s}]\in{CH_n(X\times_S X)}$$
We have $q_{j,s}=[\xi_{2m+1-j}]\circ[s]\circ[\xi_{j}]^t$. 
Similarly, one shows that $q_{j,C}q_{k,s}=q_{k,s}q_{j,C}=0$ and $q_{j,s}q_{k,t}=0$ for $C\subset{S}$, $s\in{S}$ and $t\in{S}$.  

\end{proof}

Let $f:X\rightarrow S$  be a smooth quadric bundle with relative dimension $2m-2$, $X$ and $S$ smooth quasi-projective, $dim(X)=n$.
Assume for simplicity that $S=C$ is a curve. 

Let $\xi_i\subset{X}$ be relative linear sections of $f:X\hookrightarrow \mathbb P(\mathcal{E})\rightarrow C$.

Let $F_{m-1}(X/C)\rightarrow\tilde{C}\rightarrow C$ be the Stein factorization where $F_{m-1}(X/C)$ is the relative Fano variety of ($m-1$)-dimensional linear subspaces.
Denote $\Gamma\subset{F_{m-1}(X/C)\times_C X}$ the incidence correspondence.

Let $W\subset{F_{m-1}(X/S)}$ be a smooth multisection of degree $d$ of $F_{m-1}(X/S)\rightarrow\tilde{C}$ and 
$\Gamma_W\subset{W\times_S X}=p^{-1}(W)$ the restriction of $\Gamma$ to $W$.
We assume the double covering $\tilde{C}\rightarrow S$ non trivial so that $\tilde{C}$ and $C$ are irreducible.

Denote $Z=q(\Gamma_W)\subset{X}$. $Z\times_C Z\subset{X\times_C X}$ have then two irreducible components :
$Z\times_C Z=(Z\times_C Z)^+\cup(Z\times_C Z)^-\subset{X\times_C X}$ obtained by considering diagonal and the antidiagonal of $\tilde{C}\times_S\tilde{C}$.

By the Proposition \ref{Cel}, there exists a finite covering $C'\rightarrow C$ so that the quadric bundle 
$f':X'=X\times_C C'\rightarrow C'$ obtained by base change admits a relative cellular decomposition and 
in particular contains two relative $m-1$ planes $\Lambda'$ and $\Lambda''$. Consider the relative linear sections
$\xi'_i=\pi'^{-1}(\xi_i)\subset{X'}$.
By the Proposition 5.10 we have a surjection of $\mathbb Q$ vector spaces 
$$\phi': \CH_1(C')^{\oplus 2m+2}\oplus\CH_0(C')^{\oplus 2m}\rightarrow \CH_n(X'\times_{C'} X')$$
defined by

$\phi(\alpha_0,..,\alpha_{m-1}^{11},\alpha_{m-1}^{12}\alpha_{m-1}^{21}\alpha_{m-1}^{22},..,\alpha_{2m-2},\beta_1,..,\beta'_{m-1},\beta''_{m-1},\beta'_m,\beta''_m.,\beta_{2m-2})=$\\
$([\xi'_{2m-2}\times_{C'}\xi'_0]_*(\alpha_0),..,[\Lambda'\times_{C'}\Lambda']_*(\alpha_{m-1}^{11}),[\Lambda'\times_{C'}\Lambda'']_*(\alpha_{m-1}^{12}) [\Lambda'\times_{C'}\Lambda'']_*(\alpha_{m-1}^{21}),[\Lambda'\times_{C'}\Lambda'']_*(\alpha_{m-1}^{22}),..,[\xi'_{0}\times_{C'}\xi'_{2m-2}]_*(\alpha_{2m-2}), [\xi'_{2m-2}\times_{C'}\xi'_1]_*(\beta_1),..,[\xi'_m\times_{C'}\Lambda']_*(\beta'_{m-1}),[\xi'_m\times_{C'}\Lambda'']_*(\beta''_{m-1}),
[\Lambda'\times_{C'}\xi_m]_*(\beta'_m)[\Lambda'\times_{C'}\xi_m]_*(\beta''_m),..,[\xi_{1}\times_{C'}\xi_{2m-2}]_*(\beta_{2m-2}))$

Denote by $\rho:X'\rightarrow X$ the finite covering obtained by base change.
Since $\rho\times\rho:X'\times_{C'}X'\rightarrow X\times_C X$ is surjective(it is a finite covering),
$(\rho\times\rho)_*:\CH_n(X'\times_{C'}X')\rightarrow \CH_n(X\times_C X)$ is surjective.
We thus obtain the following Corollary :

\begin{corollaire}
\label{EvenQ}

$\psi=(\rho\times\rho)_*\circ\phi': \CH_1(C')^{\oplus 2m+2}\oplus\CH_0(C')^{\oplus 2m}\rightarrow \CH_n(X\times_C X)$
is surjective.

\end{corollaire}

We now look for generators. 
We clearly have for $\beta=(\beta_1,..,c'_1,c'_2,c'_3,c'_4,..,\beta_{2m-2})\in{\CH_0(C')}$ and $c_i=\pi(c'_i)\in{C}$,
$\psi(\beta)=[\xi_{2m-2}\times_C\xi_1]_*(\pi_*(\beta_1)),..,[\xi_{m,c_1}\times\Lambda'_{c_1}],[\xi_{m,c_2}\times\Lambda''_{c_2}],
[\Lambda'_{c_3}\times\xi_{m,c_3}],[\Lambda''_{c_4}\times\xi_{m,c_4}],..,[\xi_{1}\times_C\xi_{2m-2}]_*(\pi_*(\beta_{2m-2}))$.
Denote 
\begin{equation}
\label{Sub}
K_C=\psi(\CH_0(C'))\subset{\CH_n(X\times_C X)}
\end{equation}

We clearly have for $i\neq m-1$ and $\alpha_i\in{\CH_1(C')}$
$\psi(\alpha_i)=[\xi_i\times_C\xi_{2m-2-i}]_*(\pi_*(\alpha_i))$.

\begin{lemme}
\label{Gen}

Denote 
$p_{2m-2}^+=\Gamma_W g^*g_*\Gamma^t_W=[(Z\times_C Z)^+]\in{\CH_n(X\times_C X)}$ and
$p_{2m-2}^-=\Gamma_W g^*g_*\Gamma^t_W=[(Z\times_C Z)^-]\in{\CH_n(X\times_C X)}$.

Then 
\begin{eqnarray}
(\rho\times\rho)_*(\Lambda'\times_{C'}\Lambda')&=&p_{2m-2}^++\omega_1 \\
(\rho\times\rho)_*(\Lambda''\times_{C'}\Lambda'')&=&p_{2m-2}^++\omega_2 \\
(\rho\times\rho)_*(\Lambda'\times_{C'}\Lambda'')&=&p_{2m-2}^-+\omega_3 \\ 
(\rho\times\rho)_*(\Lambda''\times_{C'}\Lambda')&=&p_{2m-2}^-+\omega_4  
\end{eqnarray}
with $\omega_i\in{K_C}$.
\end{lemme}

\begin{proof}

Let us prove the second equality. The others are similar.
To this aim, we prove that
$1/d(\rho\times\rho)^*[(Z\times_C Z)^+]-(\rho\times\rho)^*(\rho\times\rho)_*[\Lambda''\times_{C'}\Lambda'']=\omega'_2\in{\CH_n(X'\times_{C'} X')}$
with $\omega'_2\in{K_{C'}=\phi'(CH_0(C'))}$.

Let
$F=1/d\rho^*[(Z\times_C Z)^+]-\rho^*\rho_*[\Lambda''\times_{C'}\Lambda'']\in{\CH_n(X'\times_{C'} X')}$.
We have by the relative cellular decomposition of $f':X'\rightarrow C'$(Proposition 3.5)
$$F=f_{1,1}[\Lambda'\times_{C'}\Lambda']+f_{2,2}[\Lambda''\times_{C'}\Lambda'']+f_{1,2}[\Lambda'\times_{C'}\Lambda'']$$ $$+f_{2,1}[\Lambda''\times_{C'}\Lambda']+J'_f+\omega'_2\in{\CH_n(X_W\times_W X_W)}.$$

The action of $(\rho\times\rho)^*[(Z\times_C Z)^+]$ on $(R^{2m}f'_*\mathbb Q)_s=H^{2m}(X_s,\mathbb Q)$ for $s\in{C'}$ generic, the fiber of the covering 
$\pi:C'\rightarrow C$ being then of cardinal $2d$, is given by 
$[\Lambda'_s]\rightarrow d^2[\Lambda'_s]$ and
$[\Lambda''_s]\rightarrow d^2[\Lambda''_s]$
if $[\Lambda'_s].[\Lambda''_s]=0$ and $[\Lambda'_s].[\Lambda'_s]=1$. 
In this situation we have $g_*\Gamma_W\Gamma_Wg^*=d^2[D^+]\in{\CH_1(\tilde{C}\times_C\tilde{C})}$.
\\This action is given by
$[\Lambda'_s]\rightarrow d^2[\Lambda''_s]$ and
$[\Lambda''_s]\rightarrow d^2[\Lambda'_s]$
if $[\Lambda'_s].[\Lambda''_s]=1$ and $[\Lambda'_s].[\Lambda'_s]=0$.
In this situation we have $g_*\Gamma_W\Gamma_Wg^*=d^2[D^-]\in{\CH_1(\tilde{C}\times_C\tilde{C})}$.

The action of $(\rho\times\rho)^*(\rho\times\rho)_*[\Lambda''\times_{C'}\Lambda'']$ on $(R^{2m}f'_*\mathbb Q)_s=H^{2m}(X_s,\mathbb Q)$ for $s\in{C'}$ generic, the fiber of the covering $\pi:C'\rightarrow C$ being then of cardinal $2d$, is given by  
$[\Lambda'_s]\rightarrow d[\Lambda'_s]$ and
$[\Lambda''_s]\rightarrow d[\Lambda''_s]$
if $[\Lambda'_s].[\Lambda''_s]=0$ and $[\Lambda'_s].[\Lambda'_s]=1$. 
\\This action is given by
$[\Lambda'_s]\rightarrow d[\Lambda''_s]$ and
$[\Lambda''_s]\rightarrow d[\Lambda'_s]$
if $[\Lambda'_s].[\Lambda''_s]=1$ and $[\Lambda'_s].[\Lambda'_s]=0$. 
\\We have indeed 
$(\pi(\Lambda''))_s=\pi(\Lambda'')\cap X_s=\cup_{i=1}^{d}P'_{i,s}\cup\cup_{j=1}^{d}P''_{j,s}$ with $[P'_{i,s}]=[\Lambda'_s]$ and $[P''_{i,s}]=[\Lambda''_s]$,
and
$(\rho(\Lambda''\times_{C'}\Lambda'')_s=\rho(\Lambda''\times_{C'}\Lambda''))\cap X_s=\cup_{i=1}^{d}(P'_{i,s}\times P'_{i,s})\cup\cup_{j=1}^{d}(P''_{j,s}\times P''_{j,s})$.
Thus the action of $F$ on $(R^{2m}f'_*\mathbb Q)_s=H^{2m}(X_s,\mathbb Q)$ for $s\in{C'}$ generic vanishes. 

This gives $f_{1,1}=0$, $f_{2,2}=0$, $f_{1,2}=0$, $f_{2,1}=0$ and $J'_f=0$.
Thus $F=\omega'_2\in{\CH_n(X'\times_{C'} X')}$.  

Applying $(\rho\times\rho)_*$ we obtain :
$[(Z\times_C Z)^+]-d(\rho\times\rho)_*[\Lambda''\times_{C'}\Lambda'']=(\rho\times\rho)_*\omega'_2=\omega_2\in{\CH_n(X\times_S X)}$.

\end{proof}

Hence we obtain :
\begin{lemme}
\label{actnileven}

Let $F\in{\CH_n(X\times_C X)}$ a relative correspondence that acts as zero on $Rf_*\mathbb Q$.
Then $F\in{K_C}$.

\end{lemme}

\begin{proof}

By Corollary \ref{EvenQ} and Lemma \ref{Gen} there exist rational numbers $n_j$, $n^+$, $n^-$ such that
$F=\sum_{i=0,i\neq m-1}^{2m-2} n_i p_{2i}+n^+p_{2m-2}^++n^-p_{2m-2}^-+\omega$ with $\omega\in{K_C}$.
Since $F$ acts as zero on $Rf_*\mathbb Q$ we obtain $n_i=n^+=n^-=0$.

\end{proof}

We finally note that as in the odd dimensional case we have :

\begin{rem}

Let $F\in{\CH_n(X\times_C X)}$ a relative correspondence of degree zero.
If $F\in{K_C}$, then $F^2=0$.

\end{rem}





\bibliographystyle{elsarticle-num}

\providecommand{\bysame}{\leavevmode\hbox to3em{\hrulefill}\thinspace}
\providecommand{\MR}{\relax\ifhmode\unskip\space\fi MR }
\providecommand{\MRhref}[2]{
  \href{http://www.ams.org/mathscinet-getitem?mr=#1}{#2}
}
\providecommand{\href}[2]{#2}

\end{document}